\documentclass[num,STIX1COL]{WileyNJD-v2}

\articletype{RESEARCH ARTICLE}%

\received{26 April 2016}
\revised{6 June 2016}
\accepted{6 June 2016}
\usepackage{epsfig}
\usepackage{subfigure}
\newtheorem{rem}{\bf {Remark}}

\allowdisplaybreaks[3]
\usepackage{amsmath}
\usepackage{color}
\usepackage{cite}

\begin{document}
        
        \title{Delay-Adaptive  Control of  First-order Hyperbolic PIDEs}
        
        \author[1]{Shanshan Wang}
        
        \author[2]{Jie Qi}
        
        \author[3]{Miroslav Krstic}
        
        \authormark{AUTHOR ONE \textsc{et al}}

        \address[1]{\orgdiv{Department of Control Science and Engineering}, \orgname{University of Shanghai for Science and Technology}, \orgaddress{\state{Shanghai}, \country{China}}}
        
        \address[2]{\orgdiv{College of Information Science and Technology}, \orgname{Donghua University}, \orgaddress{\state{Shanghai}, \country{China}}}
        
        \address[3]{\orgdiv{Department of Mechanical and Aerospace Engineering}, \orgname{University of California, San Diego}, \orgaddress{\state{California}, \country{USA}}}
        
        \corres{Jie Qi,
                \email{jieqi@dhu.edu.cn}}
        
        %\presentaddress{This is sample for present address text this is sample for present address text}
        
        \abstract[Summary]{We develop a delay-adaptive controller for a class of first-order hyperbolic partial integro-differential equations (PIDEs) with an unknown input delay. By employing a transport PDE to represent delayed actuator states, the system is transformed into a transport partial differential equation (PDE) with unknown propagation speed cascaded with a PIDE. A parameter update law is designed using  a Lyapunov argument and the infinite-dimensional backstepping technique to establish global stability results. Furthermore, the well-posedness of the closed-loop system is analyzed. Finally, the effectiveness of the proposed method was validated through numerical simulations.}
        
        \keywords{first-order hyperbolic PIDE, delay-adaptive control, input delay, infinite-dimensional backstepping, full-state feedback}
        
        \maketitle
        
        %\footnotetext{\textbf{Abbreviations:} ANA, anti-nuclear antibodies; APC, antigen-presenting cells; IRF, interferon regulatory factor}

        \section{Introduction}
        
        First-order hyperbolic PIDEs are widely used in various engineering applications, including traffic flow \cite{2yu2019, 3bekiaris2018}, pipe flow \cite{4aamo2015}, heat exchangers \cite{5sano2003, 6ghousein2020}, and oil well drilling \cite{7mlayeh2018, 8wang2020}. These applications often involve time delays due to the transportation of matter, energy, and information, which negatively affect the stability and performance of the system. Maintaining a stable fluid temperature is critical for the normal operation of heat exchangers, but the response speed is often limited when regulating fluid temperature, resulting in a time delay \cite{sano2019boundary}. The exact value of the delay is usually hard to measure, which becomes a significant source of uncertainty within the controlled process \cite{GuptaParametric2018}. Controlling the advection process in the presence of unknown delays is, therefore, a challenging task with practical significance. Thus, addressing the stabilization problem of first-order hyperbolic PIDEs with unknown input delays is of great practical importance.
        
        Recently, there have been many studies on the stability of first-order hyperbolic PIDEs \cite{Coron2010Exact, SANO2011917}, and the development of infinite-dimensional backstepping techniques in \cite{9krstic2008book} has provided effective methods for the PDE system control problems. \cite{10krstic2008} applies this method to the control of  unstable open-loop hyperbolic PIDEs and developed a backstepping-based controller to stabilize the system. Subsequently, control problems for $2\times 2$ first-order PDEs \cite{11vazquez2011, 12Coron2013, SU2020109147}, $n+1$ coupled first-order hyperbolic PDEs \cite{13di2013}, and $m+n$ anisotropic hyperbolic systems \cite{14hu2015, 15hu2019} were investigated by employing the infinite-dimensional backstepping approach. In reference \cite{anfinsen2020}, a state feedback controller was designed for hyperbolic PIDEs with time-varying system parameters using this infinite-dimensional backstepping method, and the controller ensures that the system state converges to zero in the $H_\infty$ norm within a finite time. Furthermore, a stabilizing controller and observer for hyperbolic PIDEs with Fredholm integrals were constructed in \cite{16bribiesca2015}, and the results of \cite{17xu2017} were extended to output regulation problems. \cite{18coron} demonstrated the equivalence between finite-time stabilization and exact controllability properties for first-order hyperbolic PIDEs with Fredholm integrals. For linear anisotropic hyperbolic systems without integral terms, finite-time output regulation problems were addressed in \cite{19deutscher2017}, and stabilization problems for linear ODEs with linear anisotropic PDEs were solved in \cite{20di2018}. 
        
Infinite-dimensional backstepping has also been
                applied to adaptive control of hyperbolic PDEs. The pioneering work was presented in \cite{BERNARD20142692}, where an adaptive stabilization method was developed for a one-dimensional (1-D) hyperbolic system with a single uncertain parameter. Since then, this method has been extensively applied to various types of hyperbolic PDEs with unknown parameters, as presented in the extensive literature \cite{Yu2017Adaptive, 7917364, 23Anfinsen2019, WANG2020108640, xu2023}. The aforementioned results are built based on three traditional adaptive schemes, including the Lyapunov design, the passivity-based design,  and the swapping design, which were initially proposed for nonlinear ODEs \cite{Krstic1995}, and extended to the boundary adaptive control of PDEs \cite{Krstic2008, SMYSHLYAEV20071543, SMYSHLYAEV20071557}. Combined with backstepping design, a novel control strategy is proposed for coupled hyperbolic PDEs with multiplicative sensor faults in \cite{guo2023}, it utilized a filter-based observer and model-based fault parameter estimation technique to achieve the tracking objective.
        %Reference \cite{Yuan2022} developed a sliding mode control to address a class of complex anticollocated hyperbolic PDEs subject to multiplicative unknown faults in both the boundary sensor and actuator, achieving robust state feedback control.
        
In recent years, studies began to pay attention to the time delays that occur in first-order PIDE systems since delays are commonly encountered in engineering practice. For instance, in \cite{sano2019boundary}, input delays were considered, and a backstepping boundary control was designed for first-order hyperbolic PIDEs. An observer-based output feedback control law was proposed for a class of first-order hyperbolic PIDEs with non-local coupling terms in the domain and measurement delay compensation\cite{25qi2021}. Reference  \cite{26zhang2022}  addressed the output boundary regulation problem for a first-order linear hyperbolic PDE considering disturbances in the domain and on the boundary as well as state and sensor delays. Recently, the robustness of output feedback for hyperbolic PDEs with respect to small delays in actuation and measurements was discussed in \cite{27auriol2020}. Research on adaptive control for unknown arbitrary delays in PDE systems is relatively scarce. In contrast, there have been significant research achievements in the adaptive control of ODE systems with unknown delays. A notable theoretical breakthrough by developing adaptive control methods to compensate for uncertain actuator delays is achieved in \cite{BRESCHPIETRI20092074}. Subsequently, the delay adaptive control technique has been  applied to various types of unknown delays in ODE systems, including single-input delay \cite{Bresch2010, BEKIARISLIBERIS2010277}, multi-input delay \cite{zhu20181} and distributed input delay \cite{Zhu20201, Zhu20203}. Inspired by these studies, recent work on parabolic systems with unknown input delays is presented in \cite{wang2021adaptive, wang2021delay}. However, research on hyperbolic PDE systems with delays remains relatively limited.      For the first-order hyperbolic systems with uncertain transport speed, parameter estimators and adaptive controllers are designed in \cite{Anfinsen2017Estimation, 9183027} by using swapping filters. Different from these two studies, we apply a Lyapunov argument combined with the infinite-dimensional backstepping technique to design a delay-adaptive controller that achieves global stability in this paper, since the Lyapunov based adaptive methods are known to provide better transient performance \cite{Krstic2008}.

        In this  paper, we consider a hyperbolic PIDE with an arbitrarily large unknown input delay.   We extend the previous work on parabolic PDEs \cite{wang2021adaptive, wang2021delay} to a first-order PIDE system. We employ the infinite-dimensional backstepping method and choose the classic update law for the unknown delay, resulting in the structuring of the target system as a "cascade system", and the target transport PDE has two extra nonlinear terms which are controlled by the delay estimation  error and the delay update law. 
        The $L^2$ global stability of the target system is proven using appropriate Lyapunov functionals. The inverse Volterra/backstepping transformation establishes the norm equivalence relationship between the target system and the original one, thereby achieving $L^2$ global stability of the PDE system under the designed adaptive delay compensation controller. Furthermore, the well-posedness of the closed-loop system is analyzed.  
        
        Main contributions of this paper are:
        
        {(1) This paper develops a combined approach of the infinite-dimensional backstepping and the Lyapunov functional method for delay-adaptive control design for a class of  hyperbolic PIDEs with unknown input delay. In  \cite{wang2021adaptive}, the presence of nonzero boundary conditions in the parabolic PDE target system with unknown input delay restricts us to the local stability of the closed-loop system  with delay update law. However, we leverage the property first-order hyperbolic of the system to attain global stability of the closed-loop system.}
        
        {(2) The well-posedness of the closed-loop system is established. Due to the presence of nonlinear terms and non-zero boundary conditions in the target system, the proof of well-posedness is not straightforward. We use the semigroup method to analyze the well-posedness of the target system, and construct Lyapunov functions to establish the system's asymptotic stability in the $H^1$ norm, thereby ensuring the global existence of the classical solution. Due to the invertibility of the backstepping transformation, the equivalence between the target system and the closed-loop system can be established, so that the closed-loop system is well-posed.}

        The structure of this paper is as follows: Section \ref{2} briefly describes the design of a nonadaptive controller for the considered hyperbolic PIDE system. Section \ref{3} discusses the design of the delay-adaptive control law. Section \ref{4} is dedicated to the stability analysis of the resulting adaptive closed-loop system and the well-posedness of the closed-loop system. Section \ref{5} provides consistent simulation results to demonstrate the feasibility of our approach. The paper ends with concluding remarks in Section \ref{6}.

        \textbf{Notation:} Throughout the paper, we adopt the following notation to define the $L^2$-norm for  $\chi(x)\in L^2[0,1]$:
        \begin{align}
                &\rVert \chi\rVert^2_{L^2}=\int\nolimits_{-1}^1|\chi(x)|^2\mathrm dx,
        \end{align}
        and  set $\rVert  \chi\rVert^2=\rVert \chi\rVert^2_{L^2}$. 
        
        For any given function  $\psi (\cdot , \hat D (t))$  \begin{align}
                \frac{\partial\psi (\cdot , \hat D (t))}{\partial t}= \dot{\hat D }(t) \frac{\partial\psi(\cdot,\hat D (t))}{\partial \hat D (t)}.
        \end{align}
        
        \section{Problem Statement and non-adaptive controller}\label{2}
        Consider the first-order PIDE with an input delay $D>0$,  
        \begin{align}
                \label{equ-u0}
                &u_t(x,t)= u_{x}(x,t)+g(x) u(0,t)+\int\nolimits_0^xf(x,y)u(y,t)\mathrm dy,\\
                %\label{equ-ub0}
                &u(1,t)=U(t-D),\\
                \label{equ-ubud0}
                &u(x,0)=u_0(x),
        \end{align}
        for $ (x,t)\in (0,1)\times \mathbb{R}_+$, where  $g(x),f(x,y)\in C[0,1]$ are known coefficient functions.  Following \cite{Krstic2009}, the delayed input $U(t-D)$ is written as a transport equation coupled with \eqref{equ-u0} as follows: 
        %   Hence, system \eqref{equ-u0}--\eqref{equ-ubud0} is equivalent to
        \begin{align}
                \label{equ-ucas}
                & u_t(x,t)= u_{x}(x,t)+g(x)  u(0,t)+\int\nolimits\nolimits_0^xf(x,y)u(y,t)\mathrm dy,\\
                &u(1,t)=v(0,t),\\
                &u(x,0)=u_0(t),\\
                &Dv_t(x,t)=v_x(x,t),~~~x\in[0,1),\\
                &v(1,t)=U(x,t),\label{equ-ucas0}\\
                &v(x,0)=v_0(x),\label{equ-t0}
        \end{align}
        where the infinite-dimensional actuator state is solved as 
        \begin{align}
                v(x,t)=U(t+D(x-1)).\label{equ-VandU}
        \end{align}
        %where $v(x, t)$ is the state of the actuator, and the known propagation speed is given by $1/D$.   
        To design the delay-compensated controller $U(t)$, the backstepping transformation as follows can be employed:
        \begin{align}
                \label{equ-tranuw}
                w(x,t)=&u(x,t)-\int\nolimits\nolimits_{0}^xk(x,y) u(y,t)\mathrm dy,\\
                \label{equ-tranvz}
                z(x,t)=&v(x,t)-\int\nolimits\nolimits_{0}^1\gamma(x,y) u(y,t)\mathrm dy-D\int\nolimits\nolimits_0^xq(x-y)v(y,t)\mathrm  dy,
        \end{align}
        where the kernel function $k(x,y)$ and $q(x-y)$ are defined on $\mathcal{T}_1=\{(x,y): 0\leq y \leq x\leq 1\}$, $\gamma(x,y)$ on $\mathcal{T}_2 =\{(x,y): 0\leq y,~x \leq 1\}$, which gives the following target system 
        \begin{align}
                \label{equ-w0}
                & w_t(x,t)=w_{x}(x,t),\\
                &w(1,t)=z(0,t),\\
                & w(x,0)= w_0(x),\\
                \label{equ-z0}
                &Dz_t(x,t)=z_x(x,t),\\
                &z(1,t)=0,\label{equ-zcas0}\\
                &z(x,0)=z_0(x),\label{equ-zt0}
        \end{align}
        with a mild solution for $z$
        \begin{align} z(x,t)=
                \begin{cases}
                        z_0(x+\frac{t}{D}),             & 0\le x+\frac{t}{D}\le 1,\\
                        0,          & x+\frac{t}{D}>1,
                \end{cases} 
        \end{align}
        Using the backstepping method, one can get the kernel equations 
        \begin{align}
                k_x(x,y)=&-k_y(x,y)+\int\nolimits\nolimits_y^xf(\tau,y)k(\tau,y)\mathrm d\tau-f(x,y),\label{equ-k0}\\
                k(x,0)=&\int\nolimits\nolimits_0^x k(x,y)g(y)\mathrm dy-g(x),\\
                \gamma_x(x,y)=&-D\gamma_{y}(x,y)+D\int\nolimits\nolimits_y^1f(\tau,y)\gamma(x,\tau)\mathrm d\tau,\label{equ-gamma0}\\
                \gamma(x,0)=&\int\nolimits\nolimits_0^1g(y)\gamma(x,y)\mathrm dy,\\
                \gamma(0,y)=&k(1,y),\label{equ-gamma3}\\
                q(x)=&\gamma(x,1).\label{equ-q1}
        \end{align}
        %One can prove that equations \eqref{equ-k0}--\eqref{equ-gamma3} are well-posed following \cite{QI2021109565}, and the solution of
        %\eqref{equ-q1} is  $q(x-y)=\gamma(x-y,1)$.
        % \begin{align}
                % \nonumber&\gamma(x,y)=-\frac{\lambda'+c}{1 f(x)}\sum_{n=1}^{\infty}e^{D(\lambda_1-\frac{n^2\pi^2\varepsilon}{l^2})s}\\
                % &~~~~~~~~~~~~~~~\cdot \sin(\frac{n\pi}{1} y)\sin(\frac{n\pi}{1} x),\label{equ-gamma}\\
                % &q(x,y,r)=f(y)\gamma(x,s-r,y).\label{equ-q}
                % \end{align}
        From the boundary conditions \eqref{equ-ucas0} and \eqref{equ-zcas0}, the associated control law is straightforwardly derived
        \begin{align}
                U(t)=\int\nolimits\nolimits_{0}^{1}\gamma(1,y) u(y,t)\mathrm dy+D\int\nolimits\nolimits_0^1q(1-y)v(y,t)\mathrm dy.\label{equ-U}
        \end{align}
        %Here, the stability of the target system \eqref{equ-w0}--\eqref{equ-zt0} implies that of the original system \eqref{equ-ucas}--\eqref{equ-t0}.
        Knowing that the transformations \eqref{equ-tranuw}--\eqref{equ-tranvz} are invertible with inverse transformation as
        \begin{align}
                \label{equ-tranwu}
                u(x,t)=&w(x,t)+\int\nolimits\nolimits_{0}^xl(x,y) w(y,t)\mathrm dy,\\
                \label{equ-tranzv}
                v(x,t)=&z(x,t)+\int\nolimits\nolimits_{0}^1\eta(x,y) w(y,t)\mathrm dy-D\int\nolimits\nolimits_0^xp(x-y)z(y,t)\mathrm  dy,
        \end{align}
        where kernels $l(x,y)$, $\eta(x,y)$ and $p(x-y)$ satisfy the following PDEs,
        \begin{align}
                &l_x(x,y)+l_y(x,y)=-\int\nolimits\nolimits_y^xf(\tau,y)l(\tau,y)\mathrm d\tau-f(x,y),\label{equ-l0}\\
                &l(x,0)=-g(x),\\
                &\eta_x(x,y)+D\eta_{y}(x,y)=0,\label{equ-eta0}\\
                &\eta(x,0)=0,\\
                &\eta(0,y)=l(1,y),\label{equ-eta3}\\
                &p(x)=\eta(x,1).\label{equ-p1}
        \end{align}
        %One can prove that equations \eqref{equ-l0}--\eqref{equ-eta3} are well-posed following \cite{QI2021109565}, and the solution of \eqref{equ-p1} is  $p(x-y)=\eta(x-y,1)$.
        %The proof of global stabilization of systems similar to \eqref{equ-ucas}--\eqref{equ-t0} under the control law \eqref{equ-U} is easy to obtain. 
        Next, we will develop an adaptive controller with delay update law to stabilize \eqref{equ-ucas}--\eqref{equ-t0} for the arbitrarily long unknown delay.

        \section{Design of a Delay-Adaptive Feedback Control}\label{3}
        \subsection{Adaptive control design}
        Considering the plant \eqref{equ-u0}--\eqref{equ-ubud0} with an unknown delay $D>0$, which equivalent to the cascade system    \eqref{equ-ucas}--\eqref{equ-t0} with an unknown propagation speed  $1/D$, we will design an adaptive boundary controller to ensure global stability result.  
        \newtheorem{assuption}{\bf{Assumption}}
        \begin{assuption}\label{as1}
                The upper   and  lower bounds   $\overline D$ and $\underline{D}$   for delay $D>0$ are known.
        \end{assuption}
        Based on the certainty equivalence principle, we rewrite controller \eqref{equ-U} by replacing $D$ with estimated delay $
        \hat D(t)$ as the delay-adaptive controller
        \begin{align}
                \label{equ-hatU}
                U(x,t)=\int\nolimits\nolimits_{0}^{1}\gamma(1,y,\hat D(t)) u(y,t)\mathrm dy+\hat D(t)\int\nolimits\nolimits_0^{1}q(1-y,\hat D(t))v(y,t)\mathrm dy.
        \end{align}

        \subsection{Target system for the plant with  unknown input delay}
        %To prove the stability of the plant  \eqref{equ-u0}-\eqref{equ-ubud0},   equivalently, the system \eqref{equ-ucas}--\eqref{equ-t0} under  the control law  \eqref{equ-hatU}, 
        Rewriting the backstepping transformations \eqref{equ-tranzv}  as 
        \begin{align}
                \label{equ-unvz}
                z(x,t)=v(x,t)-\int\nolimits\nolimits_{0}^1\gamma(x,y,\hat D(t)) u(y,t)\mathrm dy-\hat D(t)\int\nolimits\nolimits_0^xq(x-y,\hat D(t))v(y,t)\mathrm dy,
        \end{align}
        and its inverse \eqref{equ-tranzv}  as:
        \begin{align}
                \label{equ-unuw1}
                v(x,t)=z(x,t)+\int\nolimits\nolimits_{0}^1\eta(x,y,\hat D(t)) u(y,t)\mathrm dy+\hat D(t)\int\nolimits\nolimits_0^xp(x-y,\hat D(t))z(y,t)\mathrm dy,
        \end{align} 
        where the kernels $\gamma(x,y,\hat D(t))$, $q(x-y,\hat D(t))$, $\eta(x,y,\hat D(t))$, $p(x-y,\hat D(t))$ satisfy the same form of PDEs \eqref{equ-k0}-\eqref{equ-q1} and \eqref{equ-l0}-\eqref{equ-p1} except $D$ replaced with $\hat D(t)$. Using the transformation \eqref{equ-tranuw} and  \eqref{equ-unvz}, we get the following target system%\newtheorem{lemma}{Lemma}
        %\begin{lemma} \label{lemma1}
        %The transformation \eqref{equ-unuw}-\eqref{equ-unvz} maps the system \eqref{equ-unu} }-\eqref{equ-unu0} into 
\begin{align}
        \label{equ-unw1}
        & w_t(x,t)= w_{x}(x,t),\\
        \label{equ-bud1}
        & w(1,t)=z(0,t),\\
        \label{equ-checku0}
        & w(x,0)= w_0(x),\\
        \label{equ-bud4}
        &Dz_t(x,t)=z_x(x,t)-\tilde{D}(t)P_{1}(x,t)-D\dot{\hat{D}}(t)P_{2}(x,t),\\
        \label{equ-bud3}
        &z(1,t)=0, \\
        &z(x,0)=z_0(x), \label{equ-zt1}
\end{align}
where $\tilde D(t)=D-\hat D(t)$ is the estimation error, functions $P_{i}(x,t),  i=1,2$ are given below:
\begin{align}
        P_{1}(x,t)=& z(0,t)M_1(x,t)+\int\nolimits\nolimits_{0}^1 w(y,t)M_2(x,y,t)\mathrm  dy,\label{equ-P1}\\
        P_{2}(x,t)=&\int\nolimits\nolimits_{0}^1 z(y,t)M_3(x,y,t)\mathrm  dy+\int\nolimits\nolimits_{0}^1w(y,t)M_4(x,y,t)\mathrm  dy,\label{equ-P2}
\end{align}
with 
\begin{align}
        M_1(x,t)=&\gamma(x,1,\hat D(t)),\\
        \nonumber M_2(x,y,t)=&\gamma(x,1,\hat D(t))l(1,y)-\gamma_{y}(x,y,\hat D(t))\\
        &+\int\nolimits\nolimits_y^1\bigg(-\gamma_y(x,\xi,\hat D(t))l(\xi,y)+\gamma(x,\xi,\hat D(t))f(\xi,y)+\int\nolimits\nolimits_\xi^1\gamma(x,\tau,\hat D(t))f(\tau,\xi)l(\xi,y)\mathrm d\tau\bigg)\mathrm d\xi,\\
        \nonumber M_{3}(x,y,t)=&q(x-y,\hat D(t))+ q_{\hat D(t)}(x-y,\hat D(t))+\hat D(t)\int\nolimits\nolimits_y^xq(x-\xi,\hat D(t))p(\xi-y,\hat D(t))\mathrm  d\xi\\
        &+\hat D(t)^2\int\nolimits\nolimits_y^xq_{\hat D(t)}(x-\xi,\hat D(t))p(\xi-y,\hat D(t))\mathrm  d\xi,\\
        \nonumber M_{4}(x,y,t)=&\gamma_{\hat D(t)}(x,y,\hat D(t))+\int\nolimits\nolimits_y^1\gamma_{\hat D(t)}(x,\xi,\hat D(t))l(\xi,y)\mathrm  d\xi+\int\nolimits\nolimits_0^xq(x-\xi,\hat D(t))\eta(\xi,y,\hat D(t))\mathrm  d\xi\\
        &+\hat D(t)\int\nolimits\nolimits_0^xq_{\hat D(t)}(x-\xi,\hat D(t))\eta(\xi,y,\hat D(t))\mathrm  d\xi.\label{equ-M4}
\end{align}
%\end{lemma}
%\begin{proof}
%Boundary conditions \eqref{equ-bud1}-\eqref{equ-bud2}, \eqref{equ-bud3} %are obviously satisfied. For \eqref{equ-unw1}, it is same as the result %form \cite{Adaptive}. Substituting \eqref{equ-unvz} into \eqref{equ-unu2}, %we get \eqref{equ-bud4}
%%\end{proof}
%Let $\overline M_1=\max_{0\leq x\leq 1,\ t\geq0}\{|M_1(x,t)|\}$, $\overline M_i=\max_{0\leq x\leq y\leq 1,\ t\geq0}\{|M_i(x,y,t)|\}$,  $i=2,3,4$.

\subsection{The parameter update law}
We choose the following update law\begin{align}
        \label{equ-law1}
        \dot{\hat D}(t)=\theta\mathrm{Proj}_{[\underline D,\overline D]}\{\tau(t)\},~~~~0<\theta<\theta^*,
\end{align}
where 
%the positive adaptation gain $\theta$ is chosen sufficiently small and 
$\tau(t)$ is given as 
\begin{align}
        \label{equ-tau}
        &\tau(t)=\frac{-b_1\int\nolimits\nolimits_0^1(1+x)z(x,t)P_{1}(x,t)\mathrm dx}{N(t)},
\end{align}
with $N(t)=\frac{1}{2}\int\nolimits\nolimits_0^1 (1+x)w (x,t)^{2}\mathrm dx+\frac{b_1}{2}\int\nolimits\nolimits_0^1(1+x)z(x,t)^{2}\mathrm d x$,~$b_1>2\bar D$ and 
\begin{align}\label{equ-thetastar}
        \theta^*=\frac{\min\{\underline D,{b_1}-2\bar D\}\min\{1,b_1\}}{2b_{1}^2L^2}.
\end{align}
The standard projection operator is defined as follows
\begin{align}
        \mathrm{Proj}_{[\underline D,\overline D]}\{\tau(t)\}=\left\{
        \begin{array}{rcl}
                0 ~~~~    &    & {\hat D(t)=\underline D~\mbox{and}~\tau(t)<0},\\
                0   ~~~~  &    & {\hat D(t)=\overline D ~\mbox{and}~\tau(t)>0},\\
                \tau(t)~~     &    & \mbox{else}.
        \end{array} \right.\label{equ-law2}
\end{align}

\begin{rem}
        The projection is used to ensure the parameters $\hat D(t)$ within the known bounds $[\underline D, \overline D]$ which cannot be viewed as a  robust tool \cite{Krstic2008}. It  prevents adaptation transients by over-limiting the size of the adaptation gain. The projection
        set can be taken conservatively and can be large, however, in order to ensure stability, the size needs to be
        inversely proportional to the adaptation gain.
\end{rem}

% 
% \begin{remm}\label{remark2}
        % \textcolor{blue}{The closed-system is well-posedness, and the proof is similar as the section 8.3 in \cite{Krstic2010}. }
        % \end{remm}

% \begin{remm}\label{remark3}
        % \textcolor{blue}{ The PDE-PDE cascade system connected through the boundary generates an unbounded input operator, which requires the $H ^ 1$ norm of the actuator state  to state the stability with the help of a Lyapunov argument. Taking the derivative of the Lyapunov function, the appearance of terms involving $z_x(1,t)\neq0$ in the  stability proof prevent the statement of a   global stability result. Similarly, for the ODE case in \cite{Krstic2009}, the global stability result can not be obtained considering the  $H ^ 1$ norm of the  actuator state  in the Lyapunov function.}
        % \end{remm}

\section{ The global stability of the closed-loop system under the delay-adaptive control}\label{4}
The  following theorem states  the global  stability result of the closed-loop  system \eqref{equ-ucas}--\eqref{equ-t0} with update law \eqref{equ-law2} and adaptive controller \eqref{equ-hatU}.

\begin{theorem} \label{theorem1}
        \rm Consider the closed-loop system consisting of the plant  \eqref{equ-ucas}--\eqref{equ-t0}, the control law \eqref{equ-hatU}, and the update law \eqref{equ-law1}--\eqref{equ-law2} under  Assumption \ref{as1}. There  exist  positive constants $\rho$, $R$ such that 
        \begin{align}
                \Psi(t)\leq R(e^{\rho\Psi(0)}-1), \quad \forall t\geq0,\label{equ-Psi}
        \end{align}
        where 
        \begin{align}\label{psi}
                \Psi(t)=&\int\nolimits\nolimits_0^1 u(x,t)^2\mathrm dx+\int\nolimits\nolimits_0^1 v(x,t)^2\mathrm dx+\tilde D(t)^2.\end{align}
        Furthermore, 
        \begin{align}
                &\lim_{t\to \infty}\max_{x\in[0,1]}|u(x,t)|=0,\\
                &\lim_{t\to \infty}\max_{x\in[0,1]}|v(x,t)|=0.
        \end{align}
        %The transformation \eqref{equ-unuw}-\eqref{equ-unvz} maps the system \eqref{equ-unu}
\end{theorem}
The global stability of the $(u, v)$-system is established by the following steps:
%\eqref{equ-ucas}--\eqref{equ-t0} under control \eqref{equ-hatU}, and the update law \eqref{equ-law1}--\eqref{equ-law2}  
\begin{itemize}
        \item We establish the norm equivalence between $(u, v)$ and $( w, z)$.
        \item We introduce a Lyapunov function to prove the global stability of the $(w, z)$-system \eqref{equ-unw1}--\eqref{equ-zt1}, and then get the stability of system $(u, v)$ by using the norm equivalence.
        \item We arrive at the regulation of states $u(x,t)$ and $v(x,t)$.
\end{itemize}

% \subsection{\textcolor{red}{Relationship between $(u, v)$ and $( w , z)$}}

\subsection{Global stability of  the closed-loop system}
First, we discuss the equivalent stability property between the plant   \eqref{equ-ucas}--\eqref{equ-t0}  and the target system \eqref{equ-unw1}--\eqref{equ-zt1}. Call now kernel functions $k(x,y),\ \gamma(x,y),\ q(x-y),\ l(x,y),\ \eta(x,y),\  $ and $p(x-y)$ are bounded by $\overline k,\ \overline \gamma,\ \overline q,\ \overline l,\ \overline \eta,$ and $\overline p$ and  in their respective domains. From \eqref{equ-tranuw}, \eqref{equ-tranvz}, \eqref{equ-tranwu}, and \eqref{equ-tranzv} it is easy to find, by using Cauchy-Schwarz inequality, that
\begin{align}
        &\rVert u(t)\rVert^2+\rVert v(t)\rVert^2\leq r_1 \rVert w (t)\rVert^2+r_2\rVert z(t)\rVert^2,\label{theo1}\\
        &\rVert  w (t)\rVert^2+\rVert z(t)\rVert^2\leq s_1\rVert u(t)\rVert^2+s_2\rVert v(t)\rVert^2,\label{theo2}
\end{align}
where $r_i$ and $s_i$, $i=1,2$ are  positive constants given by
\begin{align}
        &r_1=2+2{\overline l}^2+3\overline \eta^2,\\
        &r_2=3+3\overline D^2\overline p^2,\\
        &s_1=2+2\overline k^2+3\overline \gamma^2,\\
        &s_2=3+3\overline D^2\overline q^2.
\end{align}
Next, we prove  the global stability of the closed-loop system consisting of the  $(u, v)$-system under the control law \eqref{equ-hatU}, and the update law \eqref{equ-law1}-\eqref{equ-law2}.
Introducing a Lyapunov-Krasovskii-type function
\begin{align}
        \nonumber V_{1}(t)=&D\log (1+N(t))+\frac{\tilde D(t)^2}{2\theta},\label{equ-V}
\end{align}
where $N(t)=\frac{1}{2}\int\nolimits\nolimits_0^1 (1+x)w (x,t)^{2}\mathrm dx+\frac{b_1}{2}\int\nolimits\nolimits_0^1(1+x)z(x,t)^{2}\mathrm d x$, based on the target system \eqref{equ-unw1}--\eqref{equ-zt1} and where $b_{1}$ is a positive constant.

Taking the time derivative of \eqref{equ-V} along \eqref{equ-unw1}--\eqref{equ-zt1}, we get
\begin{align}
        \nonumber\dot V_{1}(t)=&\frac{D}{N(t)}\bigg(\int\nolimits\nolimits_0^1 (1+x)w (x,t)w_t (x,t)\mathrm dx+b_1\int\nolimits\nolimits_0^1(1+x)z(x,t)z_t(x,t)\mathrm dx\bigg)-\tilde D(t)\frac{\dot{\hat D}(t)}{\theta}\\
        \nonumber=&\frac{1}{N(t)}\bigg(D\int\nolimits\nolimits_0^1 (1+x)w (x,t)w_x(x,t)\mathrm dx+b_1\int\nolimits\nolimits_0^1(1+x)z(x,t)(z_x(x,t)-\tilde{D}(t)P_{1}(x,t) -D\dot{\hat{D}}(t)P_{2}(x,t))\mathrm dx\bigg)-\tilde D(t)\frac{\dot{\hat D}(t)}{\theta}\\
        \nonumber=&\frac{1}{N(t)}\bigg(Dw(1,t)^{2}-\frac{D}{2}w(0,t)^{2}-\frac{D}{2}\rVert w\rVert^2-\frac{b_1}{2}z(0,t)^{2}-\frac{b_1}{2}\rVert z\rVert^2\\
        &-b_1\tilde{D}(t)\int\nolimits\nolimits_0^1(1+x)z(x,t)P_{1}(x,t)\mathrm dx-b_1D\dot{\hat{D}}(t)\int\nolimits\nolimits_0^1(1+x)z(x,t)P_{2}(x,t)\mathrm dx\bigg)-\tilde D(t)\frac{\dot{\hat D}(t)}{\theta},
\end{align}
where we have used integration by parts, Cauchy-Schwarz, and Young's inequalities. 
%knowing that  $|\tilde D(t)|=|D-\hat D(t)|\leq D+\hat D(t)\leq2\overline D$.
Using \eqref{equ-law1}--\eqref{equ-thetastar} and the standard properties of the projection
operator leads to 
\begin{align}
        \nonumber\dot V_{1}(t)\leq&\frac{1}{N(t)}\bigg(-\frac{D}{2}\rVert w\rVert^2-\frac{b_1}{2}\rVert z\rVert^2-(\frac{b_1}{2}-D)z(0,t)^{2}\\
        &-b_1D\dot{\hat{D}}(t)\int\nolimits\nolimits_0^1(1+x)z(x,t)P_{2}(x,t)\mathrm dx\bigg),
\end{align}
%then, let $\iota_1=7$, $\iota_2=\frac{3}{5}$, we have 
where $b_{1}>2\bar D$.
% where $2D-\frac{1}{2\iota_1}>0$, $1-\frac{1}{2\iota_2}>0$, $1-\frac{\iota_1}{10}-\frac{\iota_2}{2}\geq0$, so $\frac{1}{4\underline D}\leq \iota_1<\frac{15}{2}$ and $\frac{1}{2}\leq \iota_2<2-\frac{1}{20\underline D}$ with $\underline D>\frac{1}{30}$. 

After a lengthy but straightforward calculation, employing the Cauchy-Schwarz and Young inequalities, along with \eqref{equ-P1} and \eqref{equ-P2}, yields the following estimates
\begin{align}
        \label{equ-ineq1}
        &\int\nolimits\nolimits_0^1(1+x)z(x,t)P_{1}(x,t)\mathrm dx\leq L(\rVert w\rVert^2+\rVert z\rVert^2+\rVert z(0,t)\rVert^2),\\
        &\int\nolimits\nolimits_0^1(1+x)z(x,t)P_2(x,t)\mathrm dx\leq L(\rVert w\rVert^2+\rVert z\rVert^2),
        \label{equ-P2xinequ}
\end{align}
where the parameter $\bar L$  is defined below 
\begin{align}
        \nonumber \bar L=\max\bigg\{\overline M_1+\overline M_2,2\overline M_3+\overline M_4\bigg\},
\end{align}
where $\overline M_1=\max_{0\leq x\leq1,\ t\geq0}\{|M_1(x,\hat D(t))|\}$, $\overline M_{i}=\max_{0\leq x\leq y\leq1,~t\geq0} \{|M_{i}(x,y,\hat D(t))|\}$ for $i=2,3,4$. 
%As an example,  below, we establish the existence of $\overline M_1$ by proving the boundedness of $M_1(x,\hat D(t))$,
% \textcolor{blue}{ \begin{align}
                % \nonumber &M_1(x,\hat D(t))=-\gamma_y(x,1,\hat D(t))\\
                % \nonumber=&-2\sum_{n=1}^{\infty}e^{\hat D(t)(\lambda-n^2\pi^2)x}n\pi (-1)^n\int\nolimits\nolimits_0^1\sin(n\pi s)k(1,s)\mathrm ds\\
                % \nonumber=&2\sum_{n=1}^{\infty}e^{\hat D(t)(\lambda-n^2\pi^2)x} (-1)^n\int\nolimits\nolimits_0^1k(1,s)\mathrm d\cos(n\pi s)\\
                % \nonumber=&2\sum_{n=1}^{\infty}e^{\hat D(t)(\lambda-n^2\pi^2)x} (-1)^n\left(k(1,1)(-1)^n\right.\\
                % \nonumber&\left.-\int\nolimits\nolimits_0^1\cos(n\pi s)k_s(1,s)\mathrm ds\right)\\
                % \nonumber =&-2\sum_{n=1}^{\infty}e^{\hat D(t)(\lambda-n^2\pi^2)x} (-1)^n\int\nolimits\nolimits_0^1\cos(n\pi s)k_s(1,s)\mathrm ds\\
                % &-\lambda\sum_{n=1}^{\infty}e^{\hat D(t)(\lambda-n^2\pi^2)x},
                % \end{align}
        % where $k(x,y)$ is a kernel defined as (10) and $k(1,1)=-\frac{\lambda}{2}$.   Due to 
        % \begin{align}
                % \lim_{n\to\infty}e^{\hat D(t)(\lambda-n^2\pi^2)x}=0,\quad \hat D(t)\in [\underline D,\overline D],
                % \end{align}
        % and
        % \begin{align}
                % (-1)^n\int\nolimits\nolimits_0^1\cos(n\pi s)k_{s}(1,s)\mathrm ds<\infty,
                % \end{align}
        %  we get $M_1(x,\hat D(t))$ is bounded \textcolor{green}{and there is exist $\overline M_1$ such that  $M_1(x,\hat D(t))\leq\overline M_1$}.}

{According to the equivalent stability property between the plant   \eqref{equ-ucas}--\eqref{equ-t0}  and the target system \eqref{equ-unw1}--\eqref{equ-zt1}}, we can get
\begin{align}
        \dot V_{1}\leq&-\left(\min\{\frac{\underline D}{2},\frac{b_1}{2}-\bar D\}\right.\left.-\frac{\theta b_{1}^2L^2}{\min\{1,b_1\}}   \right)\frac{\rVert w\rVert^2+\rVert z\rVert^2+\rVert z(0,t)\rVert^2}{N(t)}.\label{equ-dot-V1}
\end{align}
Choosing $\theta\in(0,\theta^\star)$, where $\theta^\star$ defined by \eqref{equ-thetastar}, we know $\dot V_{1}(t)\leq0$, which gives
\begin{align}
        V_{1}(t)\leq V_{1}(0),\label{equ-V_0} 
\end{align}
for all $t\geq0$.
Hence, we get the following estimates from \eqref{equ-V}:
\begin{align}
        &\rVert w\rVert^2\leq 2(e^{\frac{V_1(t)}{\underline D}}-1),\label{equ-0}\\
        &\rVert z\rVert^2\leq \frac{2}{b_1}(e^{\frac{V_1(t)}{\underline D}}-1),\label{equ-2}\\
        &\tilde D(t)\leq \frac{2\theta V_1(t)}{\underline D}.\label{equ-3}
\end{align}

Furthermore, from \eqref{equ-V}, \eqref{theo2} and \eqref{equ-0}-\eqref{equ-2}, it follows that
\begin{align}\label{61}
        \rVert u\rVert^2+\rVert v\rVert^2\leq \left(2r_1+\frac{2r_{2}}{b_1}\right)(e^{\frac{V_1(t)}{\underline D}}-1),
\end{align}
and combining \eqref{equ-3} and \eqref{61}, we get
\begin{align}
        \Psi(t)\leq\left(2r_1+\frac{2r_{2}}{b_1}+\frac{2\theta}{\underline D}\right)(e^{\frac{V_1(t)}{\underline D}}-1).\label{equ-psi-inq0}
\end{align}
So, we have bounded $\Psi(t)$ in terms of $V_1(t)$ and thus, using \eqref{equ-V_0}, in terms of  $V_{1}(0)$. Now we have to bound $V_1(0)$ in terms of $\Psi(0)$. First, from  \eqref{equ-V}, it follows that 
\begin{align}
        \nonumber V_{1}(t)=D&\mathrm{log}\bigg(1+\frac{1}{2}\int\nolimits\nolimits_0^1(1+s)w(x,t)^{2}\mathrm dx+\frac{b_1}{2}\int\nolimits\nolimits_0^1(1+s)z(x,t)^{2}\mathrm dx\bigg)+\frac{\tilde D(t)^2}{2\theta}\\
        \nonumber\leq& \bar D\rVert w\rVert^2+b_{1}\bar D\rVert z\rVert^{2}+\frac{\tilde D(t)^2}{2\theta}\\
        %\nonumber\leq&(\frac{1}{2}+2D)(\rVert  u(x,t)\rVert^2+\rVert z(x,s,t)\rVert^{2})+\frac{\tilde D^2}{2\theta}\\
        \nonumber\leq&\bar D\max\{1,b_{1}\}(s_1+s_2)\left(\rVert u\rVert^2+\rVert v\rVert^{2}\right)+\frac{\tilde D(t)^2}{2\theta}\\
        \leq&\bigg(\bar D\max\{1,b_{1}\}(s_1+s_2)+\frac{1}{2\theta}\bigg)\Psi(t),
\end{align}
leading to the following relation
\begin{align}
        V_1(0)\leq&\left(\bar D\max\{1,b_{1}\}(s_1+s_2)+\frac{1}{2\theta}\right)\Psi(0).\label{equ-psi-inq1}
\end{align}
Then, combining \eqref{equ-V_0}, \eqref{equ-psi-inq0} and \eqref{equ-psi-inq1}, we have 
\begin{align}
        \Psi(t)\leq R(e^{\rho\Psi(0)}-1), ~~~~~
\end{align}
where
\begin{align}
        &R=2r_1+\frac{2r_{2}}{b_1}+\frac{2\theta}{\underline  D},\\
        &\rho=\bar D\max\{1,b_{1}\}(s_1+s_2)+\frac{1}{2\theta},
\end{align}
so we complete the proof of the stability estimate \eqref{equ-Psi}.

        \subsection{ Pointwise boundedness and regulation of the distributed states }
        
        Now, we  ensure the regulation of the distributed states.  From  \eqref{equ-V} and  \eqref{equ-dot-V1}, we get the boundedness of $\rVert w\rVert$, $\rVert z\rVert$ and $\hat D(t)$. Knowing that
        \begin{align}
                \label{equ-w_xinteg0}
                \int\nolimits\nolimits_0^t\rVert  w(\tau)\rVert^2\mathrm d\tau\leq \sup_{0\leq\tau \leq t}N(\tau)\int\nolimits\nolimits_0^t\frac{\rVert  w(\tau)\rVert^2}{N(\tau)}\mathrm d \tau,
        \end{align}
        and using   \eqref{equ-V_0} the following inequality holds \begin{align}
                \label{equ-w_xinteg1}
                N(\tau)\leq N(0)e^{\frac{\tilde D(0)^2}{2\theta}}.
        \end{align}
        Integrating \eqref{equ-dot-V1} over $[0,\ t ]$, we have\begin{align}
                &\int\nolimits\nolimits_0^t\frac{\rVert w(\tau)\rVert^2}{N(\tau)}\mathrm d \tau\leq \frac{\bar D\log N(0)+\frac{\tilde D(0)^2}{2\theta}}{\min\bigg\{\frac{1}{2},\frac{b_1}{2}-1\bigg\}-\frac{\theta b_{1}^2L^2}{\min\{1,b_1\}}}.
                \label{equ-w_xinteg2}
        \end{align}
        Substituting \eqref{equ-w_xinteg1} and \eqref{equ-w_xinteg2} into \eqref{equ-w_xinteg0},
        we get $\rVert w\rVert$ is square integrable in time. One can establish that $\rVert z\rVert$ and $\rVert z(0,t)\rVert$ are
        square integrable in time similarly. Thus,  $\rVert P_1\rVert$ and $\rVert P_{2}\rVert$ are bounded and integrable functions of time.
        
        To prove the boundedness of  $\rVert w_{x}\rVert$, we define  
        the following  Lyapunov function 
        \begin{align}
                V_2(t)=\frac{1}{2}\int\nolimits\nolimits_0^1 (1+x)w _x(x,t)^2\mathrm dx+\frac{b_{2}D}{2}\int\nolimits\nolimits_0^1 (1+x)z_x(x,t)^2\mathrm dx,\label{equ-Vnew}
        \end{align}
        where $b_2$ is a positive constant. Using  the integration by parts, the  derivative of \eqref{equ-Vnew} with respect to time is written as 
        \begin{align}
                \label{equ-dotVnew}
                \nonumber\dot V_2(t)=&\int\nolimits\nolimits_0^1 (1+x)w _x(x,t) w _{xt}(x,t)\mathrm dx+b_{2}D\int\nolimits\nolimits_0^1 (1+x)z_x(x,t)z_{xt}(x,t)\mathrm dx\\
                \nonumber=&\int\nolimits\nolimits_0^1 (1+x)w _x(x,t) w _{xx}(x,t)\mathrm dx+b_{2}\int\nolimits\nolimits_0^1 (1+x)z_x(x,t)z_{xx}(x,t)\mathrm dx\\
                \nonumber&-b_2\tilde{D}(t)\int\nolimits\nolimits_0^1(1+x)z_x(x,t)P_{1x}(x,t)\mathrm dx-b_2D\dot{\hat{D}}(t)\int\nolimits\nolimits_0^1(1+x)z_x(x,t)P_{2x}(x,t)\mathrm dx.\\
                \nonumber=&w_{x}(1,t)^2-\frac{1}{2}w_{x}(0,t)^2-\frac{1}{2}\rVert w_{x}\rVert^2+b_2z_{x}(1,t)^2-\frac{b_2}{2}z_{x}(0,t)^2-\frac{b_2}{2}\rVert z_{x}\rVert^2\\
                &-b_2\tilde{D}(t)\int\nolimits\nolimits_0^1(1+x)z_x(x,t)P_{1x}(x,t)\mathrm dx-b_2D\dot{\hat{D}}(t)\int\nolimits\nolimits_0^1(1+x)z_x(x,t)P_{2x}(x,t)\mathrm dx.
        \end{align}
        Based on \eqref{equ-unw1}, \eqref{equ-bud1}, one can get
        \begin{align}
                \nonumber w_x(1,t)=&w_t(1,t)=z_t(0,t)\\
                =&z_x(0,t)-\tilde D(t)P_1(0,t)-D\dot{\hat D}(t)P_2(0,t),\label{equ-w-bud1}
        \end{align}
        we arrive at the following inequality
        \begin{align}
                \nonumber\dot V_2(t)\leq&-\frac{1}{2}\rVert w_{x}\rVert^2-\frac{b_{2}}{2}\rVert z_{x}\rVert^2-(\frac{b_{2}}{2}-3)z_{x}(0,t)^2+3\tilde D(t)^{2}P_1(0,t)^{2}+3D^{2}\dot{\hat D}(t)^{2}P_2(0,t)^2\\
                \nonumber&+2b_{2}\tilde D(t)^{2} P_1(1,t)^{2}+2b_{2}D^2\dot{\hat D}(t)^{2} P_2(1,t)^{2}+2b_{2}|\tilde{D}(t)|\rVert z_{x}\rVert\rVert P_{1x}(x,t)\rVert\\
                &+b_{2}D|\dot{\hat{D}}(t)|\rVert z_{x}\rVert\rVert P_{2x}(x,t)\rVert.
        \end{align}
        
        Choosing  $b_2>6$,
        we get, 
        \begin{align}
                \nonumber\dot V_2(t)\leq&-\frac{1}{2}\rVert w_{x}\rVert^2-\frac{b_{2}}{2}\rVert z_{x}\rVert^2+3\tilde D(t)^{2}P_1(0,t)^{2}+3D^{2}\dot{\hat D}(t)^{2}P_2(0,t)^2+2b_{2}\tilde D(t)^{2} P_1(1,t)^{2}\\
                \nonumber&+2b_{2}\bar D^2\dot{\hat D}(t)^{2} P_2(1,t)^{2}+2b_{2}D|\dot{\hat{D}}(t)|\rVert z_{x}\rVert\rVert P_{2x}\rVert+2b_{2}|\tilde{D}(t)|\rVert z_{x}(x,t)\rVert\rVert P_{1x}(x,t)\rVert\\
                \leq&-c_{1}V_2(t)+f_1(t)V_{2}(t)+f_2(t),
        \end{align}
        where we use Young's and Agmon's inequalities. Here,  $c_{1}=\frac{1}{2}\min\{1,\frac{1}{\overline D}\}$,  and the functions $f_1(t)$ and $f_2(t)$ are given by
        \begin{align}
                \label{equ-f1}
                f_1(t)=&b_{2}\overline D^{2}(|\dot{\hat{D}}(t)|^{2}+4),\\
                \nonumber f_2(t)=&{b_2}\rVert P_{1x}\rVert^2+b_{2}\rVert P_{2x}\rVert^{2}+12\bar D^{2}P_1(0,t)^{2}+3\bar D^{2}\dot{\hat D}(t)^{2}P_2(0,t)^2+8b_{2}\bar D^{2} P_1(1,t)^{2}\\
                &+2b_{2}D^2\dot{\hat D}(t)^{2} P_2(1,t)^{2}.
        \end{align}
        Knowing that 
        \begin{align}
                P_1(0,t)^2\leq&2\overline M_1^2z(0,t)^2+2\overline M_2^2\rVert  w \rVert^2\nonumber\\
                \leq&2\overline M_1^2(\rVert  z\rVert^2+\rVert  z_x\rVert^2)+2\overline M_2^2\rVert  w \rVert^2,\\
                P_2(0,t)^2\leq& 2\overline M_3^2\rVert  z\rVert^2+ 2\overline M_4^2\rVert  w \rVert^2,
        \end{align}
        with \eqref{equ-ineq1} and \eqref{equ-P2xinequ}, we get $ |\dot{\hat D}(t)|$, $P_1(0,t)^2$, $P_2(0,t)^2$, $P_1(1,t)^2$ and $P_2(1,t)^2$  are  integrable. Then,  
        $f_1(t)$ and $f_2(t)$ are also integrable functions of time.
        Using Lemma D.3 \cite{Krstic2010}, we get that $\rVert  w _x\rVert$ and $\rVert  z _x\rVert$ are bounded, and combing the Agmon's inequality, one can deduce  the boundedness of $w(x,t)$  and $z(x,t)$ for all $x\in[0,1]$. 
        
        Next we establish the boundedness of  $\frac{\mathrm d}{\mathrm dt}(\rVert  w\rVert^2)$, $\frac{\mathrm d}{\mathrm dt}(\rVert z\rVert^2)$ and $\frac{\mathrm d}{\mathrm dt}(\rVert z_{x}\rVert^2)$ using the following  Lyapunov function 
        \begin{align}
                V_3(t)=\frac{1}{2}\int\nolimits\nolimits_0^1 (1+x)w_{x} (x,t)^2\mathrm dx+\frac{b_{3}D}{2}\int\nolimits\nolimits_0^1 (1+x)z(x,t)^2\mathrm dx+\frac{b_{3}D}{2}\int\nolimits\nolimits_0^1 (1+x)z_x(x,t)^2\mathrm dx,\label{equ-Vnew3}
        \end{align}
        where $b_3$ is a positive constant. Taking the derivative of \eqref{equ-Vnew3} with respect to time, 
        we obtain
        \begin{align}
                \label{equ-dotVnew3}
                \nonumber\dot V_3(t)=&\int\nolimits\nolimits_0^1 (1+x)w_x(x,t) w _{xt}(x,t)\mathrm dx+b_{3}D\int\nolimits\nolimits_0^1 (1+x)z(x,t)z_{t}(x,t)\mathrm dx\\
                \nonumber&+b_{3}D\int\nolimits\nolimits_0^1 (1+x)z_x(x,t)z_{xt}(x,t)\mathrm dx\\
                \nonumber=&w_{x}(1,t)^2-\frac{1}{2}w_{x}(0,t)^2-\frac{1}{2}\rVert w_{x}\rVert^2-\frac{b_3}{2}z(0,t)^2-\frac{b_3}{2}\rVert z\rVert^2+b_{3}z_{x}(1,t)^2-\frac{b_{3}}{2}z_{x}(0,t)^2\\
                \nonumber&-\frac{b_{3}}{2}\rVert z_{x}\rVert^2-b_3\tilde{D}(t)\bigg(\int\nolimits\nolimits_0^1(1+x)z(x,t)P_{1}(x,t)\mathrm dx+\int\nolimits\nolimits_0^1(1+x)z_{x}(x,t)P_{1x}(x,t)\mathrm dx\bigg)\\
                &-b_3D\dot{\hat{D}}(t)\bigg(\int\nolimits\nolimits_0^1(1+x)z(x,t)P_{2}(x,t)\mathrm dx+\int\nolimits\nolimits_0^1(1+x)z_{x}(x,t)P_{2x}(x,t)\mathrm dx\bigg).
        \end{align}
        Clearly, using integrations by part and Young's inequality, the following holds,
        \begin{align}
                \nonumber\dot V_3(t)\leq&-\frac{1}{2}\rVert  w_{x}\rVert^2-\frac{b_3}{2}\rVert z\rVert^2-(\frac{b_3}{2}-1)z(0,t)^2-\frac{b_3}{2}\rVert z_{x}\rVert^2-(\frac{b_{3}}{2}-3)z_{x}(0,t)^2\\
                \nonumber&+3\tilde D(t)^{2}P_1(0,t)^{2}+3D^{2}\dot{\hat D}(t)^{2}P_2(0,t)^2+2b_{3}\tilde D(t)^{2} P_1(1,t)^{2}+2b_{3}D^2\dot{\hat D}(t)^{2} P_2(1,t)^{2}\\
                \nonumber&+2b_3|\tilde{D}(t)|\rVert z\rVert \rVert P_{1}\rVert+2b_{3}\overline D|\dot{\hat{D}}(t)|\rVert z\rVert \rVert P_{2}\rVert+2b_3|\tilde{D}(t)|\rVert z_x\rVert\rVert P_{1x}\rVert\\
                &+2b_{3}\overline D|\dot{\hat{D}}(t)|\rVert z_{x}\rVert \rVert P_{2x}\rVert.
        \end{align}
        Choosing $b_3>6$,  we have
        \begin{align}
                \nonumber\dot V_3(t)\leq&-\frac{1}{2}\rVert  w_{x}\rVert^2-\frac{b_3}{2}\rVert z\rVert^2-\frac{b_3}{2}\rVert z_{x}\rVert^2+3\tilde D(t)^{2}P_1(0,t)^{2}+3D^{2}\dot{\hat D}(t)^{2}P_2(0,t)^2+2b_{3}\tilde D(t)^{2} P_1(1,t)^{2}\\
                \nonumber&+3D^{2}\dot{\hat D}(t)^{2}P_2(0,t)^2+2b_{3}\tilde D(t)^{2} P_1(1,t)^{2}+2b_{3}D^2\dot{\hat D}(t)^{2} P_2(1,t)^{2}+2b_3|\tilde{D}(t)|\rVert z\rVert \rVert P_{1}\rVert\\
                \nonumber&+2b_{3}\overline D|\dot{\hat{D}}(t)|\rVert z\rVert \rVert P_{2}\rVert+2b_3|\tilde{D}(t)|\rVert z_x\rVert\rVert P_{1x}\rVert+2b_{3}\overline D|\dot{\hat{D}}(t)|\rVert z_{x}\rVert \rVert P_{2x}\rVert\\
                \leq&-c_{1}V_3(t)+f_3(t)V_{3}(t)+f_4(t)<\infty,\label{wt}
        \end{align}
        where we use Young's and Agmon's inequalities, and
        \begin{align}
                \label{equ-f3}
                f_3(t)=&2b_{3}\overline D^{2}(|\dot{\hat{D}}(t)|^{2}+4),\\
                \nonumber f_4(t)=&3\tilde D(t)^{2}P_1(0,t)^{2}+3\bar D^{2}\dot{\hat D}(t)^{2}P_2(0,t)^2+{b_3}(\rVert P_{1}\rVert^2+\rVert P_{2}\rVert^{2}+\rVert P_{1x}\rVert^2+\rVert P_{2x}\rVert^{2}\\
                &+8\bar D^{2} P_1(1,t)^{2}+2\bar D^2\dot{\hat D}(t)^{2} P_2(1,t)^{2}),
        \end{align}
        are bounded functions. Thus, from \eqref{wt}, one can deduce   the boundedness of  $\frac{\mathrm d}{\mathrm dt}(\rVert  w\rVert^2)$, $\frac{\mathrm d}{\mathrm dt}(\rVert z\rVert^2)$ and $\frac{\mathrm d}{\mathrm dt}(\rVert z_{x}\rVert^2)$. Moreover, by Lemma D.2 \cite{Krstic2010}, we get $V_3(t)\to0$, and thus   $\rVert w_x                        \rVert,\rVert z\rVert, \rVert z_{x}\rVert^2\to 0$ as $t\to\infty$. Next, from \eqref{theo1}, we have  $\rVert u_{x}\rVert^2, \rVert v\rVert^2, \rVert v_{x}\rVert^2\to 0$ as $t\to\infty$.

        From \eqref{equ-tranwu}, we have 
        \begin{align}
                &\rVert u_x\rVert^2\leq 2\rVert w_x\rVert^2+2\rVert w\rVert^2\overline l_x^2.
                %&\rVert u\rVert^2\leq 2(1+Q_{1}^{2})\rVert w\rVert^2.
        \end{align}
        Since $\rVert w\rVert,\ \rVert w_x\rVert$ are bounded, $\rVert u_x\rVert$ is also bounded. By Agmon's inequality $u(x,t)^2\leq2\rVert u\rVert\rVert u_x\rVert$, which enables one to state the regulation of $u(x,t)$ to zero uniformly in $x$. Similarly, one can prove the regulation of $v(x,t)$. Since $\rVert v\rVert^2$ and $\rVert v_x\rVert^2\to 0$ as $t\to\infty$, by Agmon's inequality $v(x,t)^2\leq2\rVert v\rVert\rVert v_x\rVert$, which enables one to state the regulation of $v(x,t)$ to zero uniformly in $x$ and completes the proof of Theorem \ref{theorem1}. 
        
        \subsection{Well-posedness of the closed-loop system}
        Following the approach in   \cite{Krstic2010}, we prove  the well-posedness of the closed-loop system in Theorem 1.
        Consider the closed-loop target system $(w,z,\tilde D(t))$:
        \begin{align}
                & w_t(x,t)= w_{x}(x,t),\label{equ-checkw-well}\\
                & w(1,t)=z(0,t),\\
                &z_t(x,t)=\frac{1}{D}z_x(x,t)-\frac{\tilde{D}(t)}{D} P_1(x,t)-\theta\mathrm{Proj}_{[\underline D,\bar D]}\{\tau(t)\} P_2(x,t),\\
                &z(1,t)=0,\\ 
                &\dot{\tilde D}(t)=-\theta\mathrm{Proj}_{[\underline D,\bar D]}\{\tau(t)\},\label{equ-hatD-well}
        \end{align}
        we set $Z=( w,z,\tilde D(t))^{T}$, and  introduce the operator
        \begin{align}
                A=\begin{pmatrix}
                        -\frac{\partial}{\partial x}&0&0\\
                        0&-\frac{\partial}{D\partial x}&0\\
                        0&0&0
                \end{pmatrix},
        \end{align}
        with
        \begin{align}
                F(Z)=\begin{pmatrix}
                        0\\
                        -\frac{\tilde{D}(t)}{D} P_1(x,t)-\theta\mathrm{Proj}_{[\underline D,\bar D]}\{\tau(t)\} P_2(x,t)\\
                        \theta\mathrm{Proj}_{[\underline D,\bar D]}\{\tau(t)\}
                \end{pmatrix}.
        \end{align} 
        Then \eqref{equ-checkw-well}-\eqref{equ-hatD-well}
        can be written in abstract form as
        \begin{align}
                &Z_t=-AZ+F(Z),\label{equ-Zt}\\
                &Z(0)=Z_0.\label{equ-Z0}
        \end{align}
        where $Z=L^2(0,1)\times L^2(0,1)\times\mathbb R$,  $\mathcal{B}(A)=\{f,g,l:f\in H^1(0,1), f(1)=g(0); g\in H^{1}(0,1),g(1)=0;  l\in\mathbb R\}$ and the norm $\rVert Z\rVert_H=\rVert  w\rVert^2+\rVert z\rVert^2+\tilde D^2$. 
        
        Now, we establish the well-posedness of \eqref{equ-Zt}--\eqref{equ-Z0} with the following theorem (see as well Theorem 8.2 \cite{Krstic2010},
        Theorem 2.5.6 \cite{Zheng2004book}, for which a similar method has been
        employed to establish well-posedness).
        
        \begin{theorem} \label{theorem2}
                \rm
                Consider the system  \eqref{equ-Zt}--\eqref{equ-Z0} , where $A$ is a maximal accretive operator from a dense subset
                $\mathcal{B}(A)$ in a Banach space  $H$ into $H$. If $F$
                is a nonlinear operator from $\mathcal{B}(A)$ to $\mathcal{\mathcal{B}}(A)$ and satisfies
                the local Lipschitz condition, then for any $Z_0 \in\mathcal{B}(A)$, the
                problem  \eqref{equ-Zt}--\eqref{equ-Z0} admits a unique classical solution Z
                such that
                \begin{align}
                        Z\in C^1([0,T_{\max}),H)\cap C([0,T_{\max}),\mathcal{B}(A)),
                \end{align}
                where\\
                (i) either $T_{max}=+\infty$, i,e., there is a unique global classical solution\\ 
                (ii) or $T_{max}<+\infty$ and $\lim_{t\to T_{max}-0}\rVert Z(t)\rVert_H=+\infty$. 
        \end{theorem}
        \begin{proof}
                \rm
                Combining the proof for hyperbolic case (see, e.g., Example 2.3.1 in \cite{Zheng2004book}), we  obtain that $A$ is a maximal accretive operator. Then, it is straightforward to establish that for any $Z_1,\ Z_2 \in  H$,
                \begin{align}
                        \rVert F(Z_1)-F(Z_2)\rVert_H\leq C\rVert Z_1-Z_2\rVert_H\max\{\rVert Z_1\rVert_H,\rVert Z_2\rVert_H\},
                \end{align}
                where $C$ is a constant independent of $Z_1$ and $Z_2$. So, we get
                $F$ to be  locally Lipschitz on $H$. Hence, the system  \eqref{equ-Zt}--\eqref{equ-Z0} has a unique classical solution.
                
                Next, we will establish that  the existence of the classical
                solution is global. In order to prove that $T_{\max} = +\infty$, which means
                there is no blowup, we need to make a priori estimates of the $H^1$
                norm of $w$ {and $z$}. Based on the proof of boundedness of $w$ and $z$ in
                $L^2$ norms, in our present work, one can obtain that
                $w$ {and $z$ are} bounded in $H^1$ by using the following new Lyapunov function
                \begin{align}\label{equ-V4}
                        V_4(t)=\frac{1}{2}\int\nolimits\nolimits_0^1 (1+x)w _{xx}(x,t)^2\mathrm dx+\frac{b_{4}D}{2}\int\nolimits\nolimits_0^1 (1+x)z_{xx}(x,t)^2\mathrm dx.
                \end{align}
                Using  the integration by parts, the  derivative of \eqref{equ-V4} with respect to time is written as 
                \begin{align}
                        \label{equ-dotV4}
                        \nonumber\dot V_4(t)=&\int\nolimits\nolimits_0^1 (1+x)w_{xx}(x,t) w _{xxt}(x,t)\mathrm dx+b_{4}D\int\nolimits\nolimits_0^1 (1+x)z_{xx}(x,t)z_{xxt}(x,t)\mathrm dx\\
                        \nonumber=&\int\nolimits\nolimits_0^1 (1+x)w _{xx}(x,t) w _{xxx}(x,t)\mathrm dx+b_{4}\int\nolimits\nolimits_0^1 (1+x)z_{xx}(x,t)z_{xxx}(x,t)\mathrm dx\\
                        \nonumber&-b_4\tilde{D}(t)\int\nolimits\nolimits_0^1(1+x)z_{xx}(x,t)P_{1xx}(x,t)\mathrm dx-b_4D\dot{\hat{D}}(t)\int\nolimits\nolimits_0^1(1+x)z_{xx}(x,t)P_{2xx}(x,t)\mathrm dx\\
                        \nonumber=&w_{xx}(1,t)^2-\frac{1}{2}w_{xx}(0,t)^2-\frac{1}{2}\rVert w_{xx}(x,t)\rVert^2+b_4z_{xx}(1,t)^2-\frac{b_4}{2}z_{xx}(0,t)^2-\frac{b_4}{2}\rVert z_{xx}(x,t)\rVert^2\\
                        &-b_4\tilde{D}(t)\int\nolimits\nolimits_0^1(1+x)z_{xx}(x,t)P_{1xx}(x,t)\mathrm dx-b_4D\dot{\hat{D}}(t)\int\nolimits\nolimits_0^1(1+x)z_{xx}(x,t)P_{2xx}(x,t)\mathrm dx.
                \end{align}
                Based on \eqref{equ-unw1}, \eqref{equ-bud1}, one can get
                \begin{align}
                        \nonumber w_{xx}(1,t)=&w_{tx}(1,t)=w_{tt}(1,t)=z_{tt}(0,t)\\
                        \nonumber=&\frac{1}{D^2}z_{xx}(0,t)-\frac{\tilde D(t)}{D^2}P_{1x}(0,t)-\frac{\dot{\hat D}(t)}{D}P_{2x}(0,t)+\frac{1}{D}\dot{\hat D}(t)P_1(0,t)-\ddot{\hat D}(t)P_2(0,t)\\
                        &-\frac{1}{D}\tilde D(t)P_{1t}(0,t)-\dot{\hat D}(t)P_{2t}(0,t),\label{equ-wxx-1}\\
                        \nonumber z_{xx}(1,t)=&{\tilde D(t)}P_{1x}(1,t)+D\dot{\hat D}(t)P_{2x}(1,t)-D\dot{\hat D}(t)P_1(1,t)+D^{2}\ddot{\hat D}(t)P_2(1,t)\\
                        &+D\tilde D(t)P_{1t}(1,t)+D^{2}\dot{\hat D}(t)P_{2t}(1,t).\label{equ-zxx-1}
                \end{align}
                % and 
                % \begin{align}
                        % Dz_{tx}(x,t)=&z_{xx}(x,t)-\tilde D(t)P_{1x}(x,t)-D\dot{\hat D}(t)P_{2x}(x,t),
                        % \end{align}
                % \begin{align}
                        % \nonumber w_{xx}(1,t)=&z_{xx}(0,t)-\tilde D(t)P_{1x}(0,t)-D\dot{\hat D}(t)P_{2x}(0,t)\\
                        % \nonumber&+\dot{\hat D}(t)P_1(0,t)-D\ddot{\hat D}(t)P_2(0,t)\\
                        % &-\tilde D(t)P_{1t}(0,t)-D\dot{\hat D}(t)P_{2t}(0,t),\label{equ-w-bud1}
                        % \end{align}
                % \begin{align}
                        % \nonumber z_{xx}(1,t)=&\tilde D(t)P_{1x}(x,t)+D\dot{\hat D}(t)P_{2x}(x,t)\\
                        % \nonumber&-\dot{\hat D}(t)P_1(x,t)+D\ddot{\hat D}(t)P_2(x,t)\\
                        % &+\tilde D(t)P_{1t}(x,t)+D\dot{\hat D}(t)P_{2t}(x,t),\label{equ-w-bud1}
                        % \end{align}
                % and 
                % \begin{align}
                        % z_{tx}(x,t)=&z_{xx}(x,t)-\tilde D(t)P_{1x}(x,t)-D\dot{\hat D}(t)P_{2x}(x,t),
                        % \end{align}
                Submitting \eqref{equ-wxx-1} and \eqref{equ-zxx-1} into \eqref{equ-dotV4},
                we arrive at the following inequality
                \begin{align}
                        \nonumber\dot V_4(t)\leq&-\frac{1}{2}w_{xx}(0,t)^2-\frac{1}{2}\rVert w_{xx}\rVert^2-\frac{b_4}{2}\rVert z_{xx}\rVert^2-\bigg(\frac{b_4}{2}-\frac{7}{D^4}\bigg)z_{xx}(0,t)^2+2b_4|\tilde{D}(t)|\rVert z_{xx}\rVert \rVert P_{1xx}\rVert \\
                        \nonumber&+2b_4D|\dot{\hat{D}}(t)|\rVert z_{xx}\rVert \rVert P_{2xx}\rVert +\frac{7\tilde D(t)^2}{D^4}P_{1x}(0,t)^{2}+\frac{7\dot{\hat D}(t)^2}{D^{2}}P_{2x}(0,t)^{2}+\frac{7\dot{\hat D}(t)^{2}}{D^{2}}P_1(0,t)^{2}\\
                        \nonumber&+7\ddot{\hat D}(t)^2P_2(0,t)^{2}+\frac{7}{D^2}\tilde D(t)^2P_{1t}(0,t)^2+7\dot{\hat D}(t)^2P_{2t}(0,t)^2+6b_4{\tilde D(t)^2}P_{1x}(1,t)^2\\
                        \nonumber&+6b_4D^2\dot{\hat D}(t)^2P_{2x}(1,t)^2+6b_4D^2\dot{\hat D}(t)^2P_1(1,t)^2+6b_4D^{4}\ddot{\hat D}(t)^2P_2(1,t)^{2}+6b_4D^2\tilde D(t)^2P_{1t}(1,t)^2\\
                        &+6b_4D^{4}\dot{\hat D}(t)^2P_{2t}(1,t)^2.
                \end{align}
                Choosing  $b_4>\frac{14}{\underline D^4}$,
                we get, 
                \begin{align}
                        \nonumber\dot V_4(t)\leq&-\frac{1}{2}\rVert w_{xx}\rVert^2-\frac{b_4}{2}\rVert z_{xx}\rVert^2+b_4\tilde{D}(t)^2\rVert z_{xx}\rVert^{2}+b_4 \rVert P_{1xx}\rVert^{2} +b_4D^{2}\dot{\hat{D}}(t)^2\rVert z_{xx}\rVert^2 +b_4\rVert P_{2xx}\rVert^2\\
                        \nonumber&+\frac{7\tilde D(t)^2}{D^4}P_{1x}(0,t)^{2}+\frac{7\dot{\hat D}(t)^2}{D^{2}}P_{2x}(0,t)^{2}+\frac{7\dot{\hat D}(t)^{2}}{D^{2}}P_1(0,t)^{2}+7\ddot{\hat D}(t)^2P_2(0,t)^{2}\\
                        \nonumber&+\frac{7}{D^2}\tilde D(t)^2P_{1t}(0,t)^2+7\dot{\hat D}(t)^2P_{2t}(0,t)^2+6b_4{\tilde D(t)^2}P_{1x}(1,t)^2+6b_4D^2\dot{\hat D}(t)^2P_{2x}(1,t)^2\\
                        \nonumber&+6b_4D^2\dot{\hat D}(t)^2P_1(1,t)^2+6b_4D^{4}\ddot{\hat D}(t)^2P_2(1,t)^{2}+6b_4D^2\tilde D(t)^2P_{1t}(1,t)^2+6b_4D^{4}\dot{\hat D}(t)^2P_{2t}(1,t)^2\\
                        \leq&-c_{1}V_4(t)+f_5(t)V_{4}(t)+f_6(t),
                \end{align}
                where we use Young's and Agmon's inequalities. Here, $c_{1}=\frac{1}{2}\min\{1,\frac{1}{\overline D}\}$,  and the functions $f_5(t)$ and $f_6(t)$ are given by
                \begin{align}
                        f_5(t)=&b_{5}\overline D^{2}(\dot{\hat{D}}(t)^{2}+4),\label{equ-f5}\\
                        \nonumber f_6(t)=&b_4 \rVert P_{1xx}\rVert^{2}+b_4\rVert P_{2xx}\rVert^2+\frac{28\bar  D^2}{\underline D^4}P_{1x}(0,t)^{2}+\frac{7\dot{\hat D}(t)^2}{\underline D^{2}}P_{2x}(0,t)^{2}+\frac{7\dot{\hat D}(t)^{2}}{\underline D^{2}}P_1(0,t)^{2}\\
                        \nonumber&+7\ddot{\hat D}(t)^2P_2(0,t)^{2}+\frac{28}{D^2}\bar D^2P_{1t}(0,t)^2+7\dot{\hat D}(t)^2P_{2t}(0,t)^2+24b_4{\bar D^2}P_{1x}(1,t)^2\\
                        \nonumber&+6b_4\bar D^2\dot{\hat D}(t)^2P_{2x}(1,t)^2+6b_4\bar D^2\dot{\hat D}(t)^2P_1(1,t)^2+6b_4\bar D^{4}\ddot{\hat D}(t)^2P_2(1,t)^{2}+24b_4\bar D^4P_{1t}(1,t)^2\\
                        &+6b_4\bar D^{4}\dot{\hat D}(t)^2P_{2t}(1,t)^2,\label{equ-f6}
                \end{align}
                based on all above results, we can get all terms in \eqref{equ-f5} and \eqref{equ-f6} are  integrable of time.
                Using Lemma D.3 \cite{Krstic2010}, we get that $\rVert  w_{xx}\rVert$ and $\rVert  z_{xx}\rVert$ are bounded. Then, from \eqref{equ-unw1} and \eqref{equ-bud3}, \begin{align}
                        \label{equ-w-tx}
                        & w_{tx}(x,t)= w_{xx}(x,t),\\
                        \label{equ-z-tx}
                        &Dz_{tx}(x,t)=z_{xx}(x,t)-\tilde{D}(t)P_{1x}(x,t)-D\dot{\hat{D}}(t)P_{2x}(x,t),
                \end{align}              
                we get
                $\rVert  w_{tx}\rVert$ and $\rVert  z_{tx}\rVert$ are bounded. Combing  with  $\rVert w_x\rVert,\rVert z_{x}\rVert^2\to 0$ as $t\to\infty$ and  regulation of $w(x,t)$ and $z(x,t)$, one can get   $\rVert w_t                        \rVert,\rVert z_{t}\rVert^2\to 0$ as $t\to\infty$, and then, by using the Agmon's inequality,   the  regulation  of $w_t(x,t)$  and $z_t(x,t)$ is proven for all $x\in[0,1]$. Therefore, we have proved that $\rVert Z\rVert_H$ is bounded and global
                classical solution exists.
        \end{proof} 
        
        Finally, we can get the well-posedness of the closed-loop system  consisting of the plant  \eqref{equ-ucas}--\eqref{equ-t0}, the control law \eqref{equ-hatU}, and the update law \eqref{equ-law1}--\eqref{equ-law2} the under Assumption 1 based on the invertibility of the backstepping transformations \eqref{equ-tranwu} and \eqref{equ-unuw1}.

        \section{Simulation}\label{5}
        To illustrate the feasibility of the proposed adaptive controller design, we simulate the closed-loop system  consisting  \eqref{equ-ucas}--\eqref{equ-ucas0},  the control law \eqref{equ-U}, and the  update law defined through \eqref{equ-law1}--\eqref{equ-law2}. 
        The actual delay is set to  $D = 2$ assuming known upper  and lower bounds defined as  $\bar D = 4$ and $\underline D=0.1$, respectively. 
        The adaptation gain is set to $\theta = 0.021$, the plant coefficients are chosen as $g(x)=2(1-x)$ and  $f(x,y)=\cos(2\pi x)+4\sin(2\pi y))$. The simulations are performed considering    $u_0(x)=4\sin(\pi x),\ v_{0}(x)=0$ as initial conditions with  $\hat D_{0}=1$ and $\hat D_{0}=3$, respectively.
        Figure \ref{figure} shows  the convergence of  the  plant's state $u(x,t)$ with and without adaptation, respectively.  In the absence of adaptation, but with a "mismatch input delay" set to
        $\hat D(t) = 3$ (the true delay being $D = 2$). Figure \ref{figure2} (a) shows the dynamics of the  $L^2$-norm of the plant state $\|u(x,t)\|_{L^2}$  with and without adaptation, respectively.  The   control effort is displayed in Figure \ref{figure2} (b) and the update law  in Figure \ref{figure2} (c). Finally,  Figure \ref{figure2} (d)  reflects a good estimate  of the delay  with $\hat D(t)$ converging to the true value $D=2$.
        \begin{figure}[htbp]
                \centering
                \subfigure[]{\includegraphics[width=0.32\textwidth]{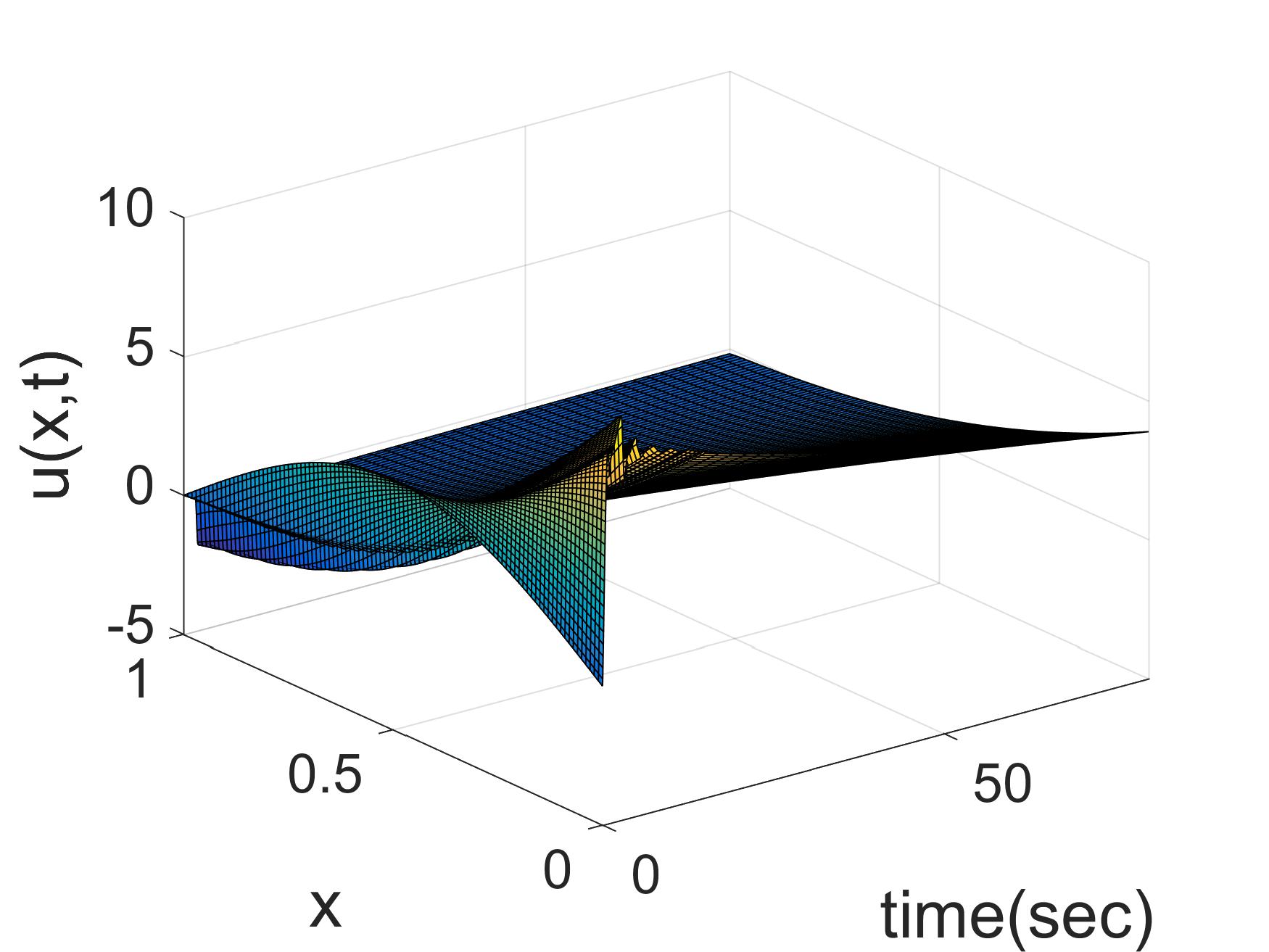}}
                \subfigure[]{\includegraphics[width=0.32\textwidth]{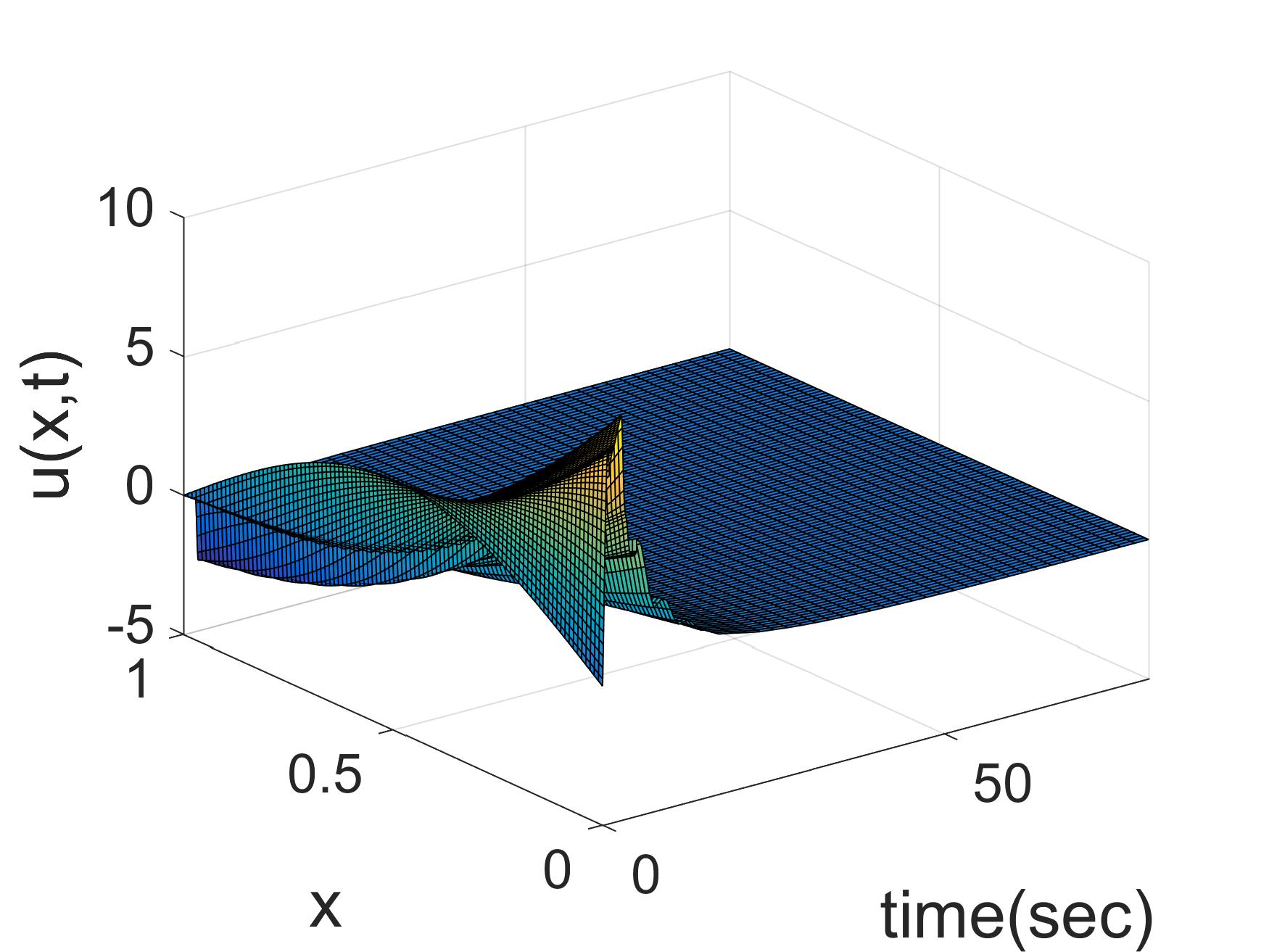}}
                \subfigure[]{\includegraphics[width=0.32\textwidth]{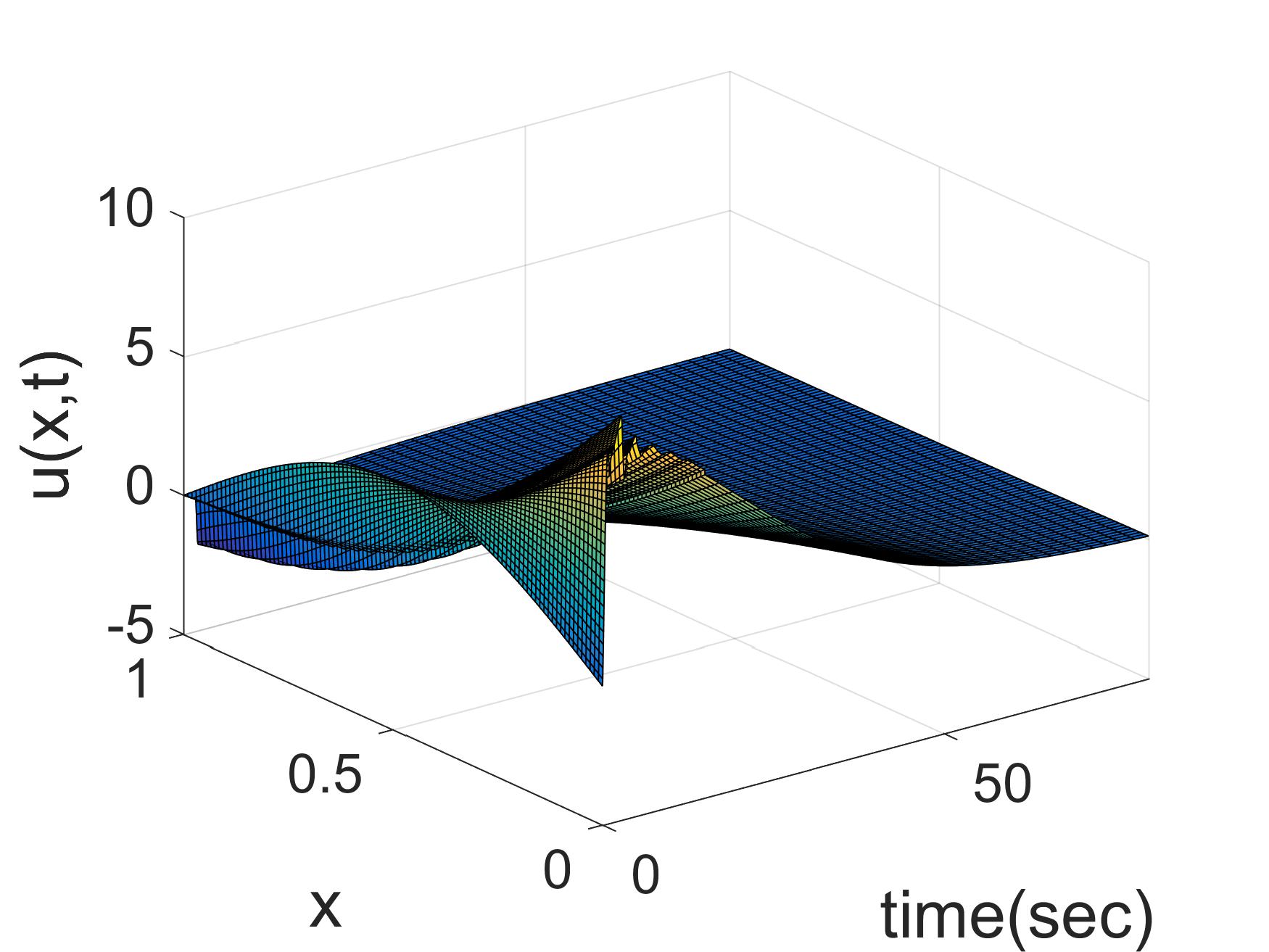}}
                \caption{The closed-loop system dynamics  with $u_0(x)$, $v_0(x)$ and $\hat D(0)$. (a) The distributed state $u(x,t)$ with nonadaptive control. (b) The distributed state $u(x,t)$ with $\hat D(0)=1$. (c) The distributed state $u(x,t)$ with $\hat D(0)=3$.}\label{figure}
        \end{figure}
        \begin{figure}[htbp]
                \centering
                \subfigure[]{\includegraphics[width=0.495\textwidth]{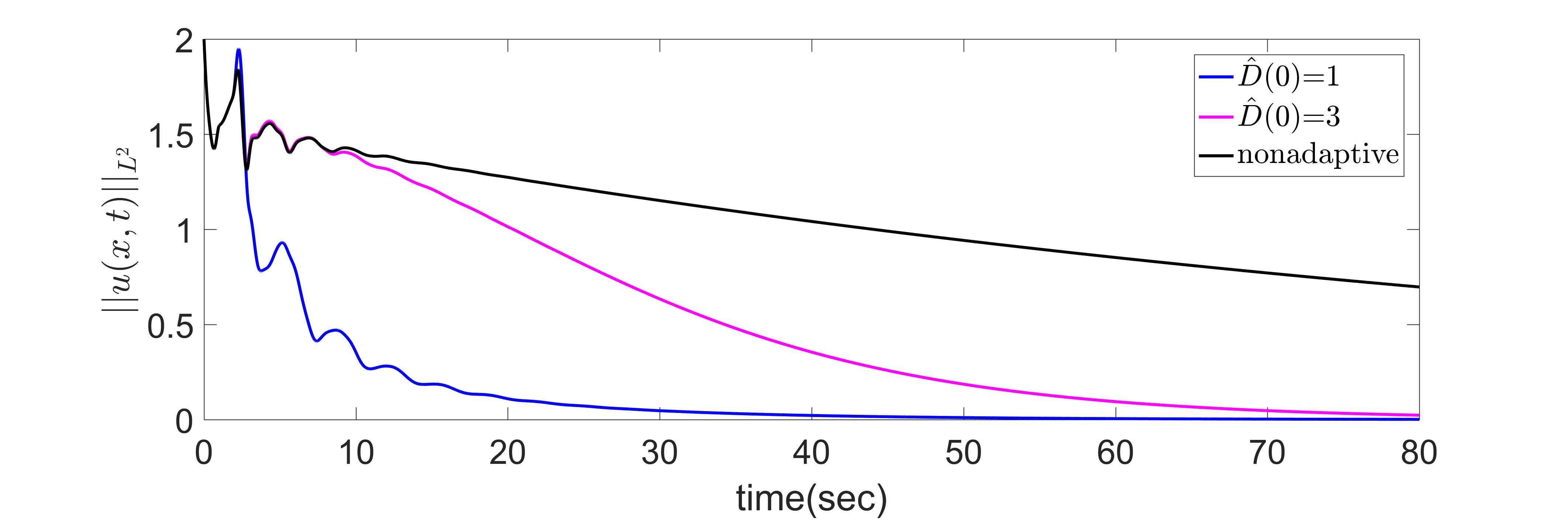}}
                \subfigure[]{\includegraphics[width=0.495\textwidth]{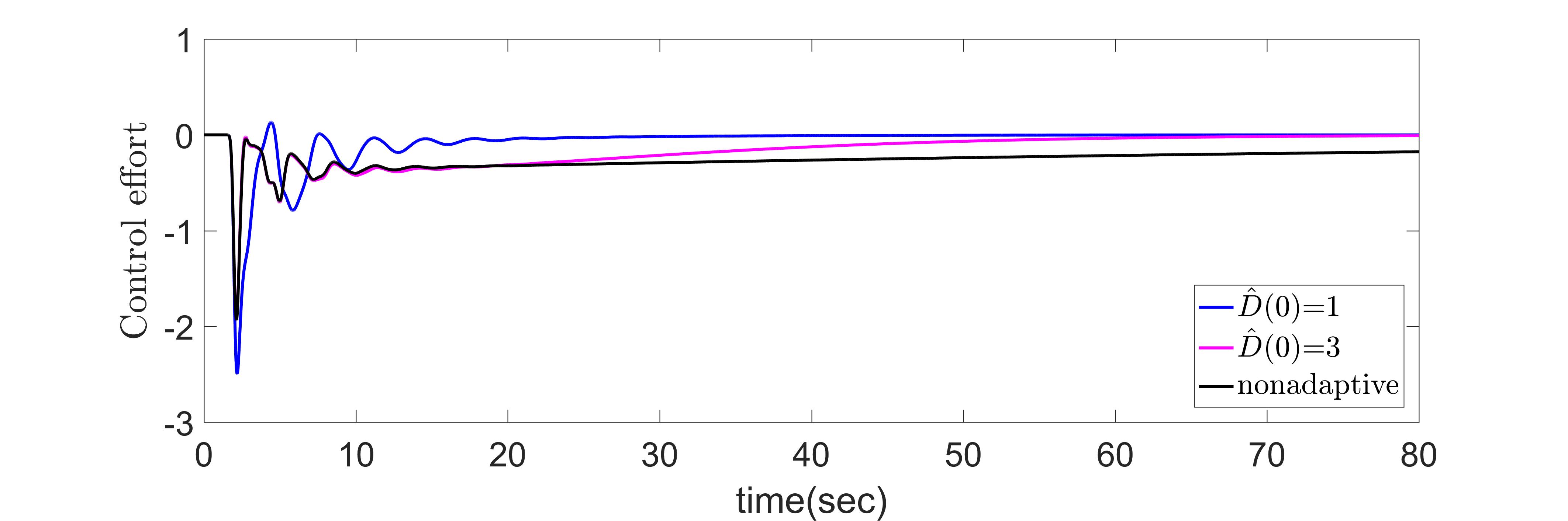}}
                \subfigure[]{\includegraphics[width=0.495\textwidth]{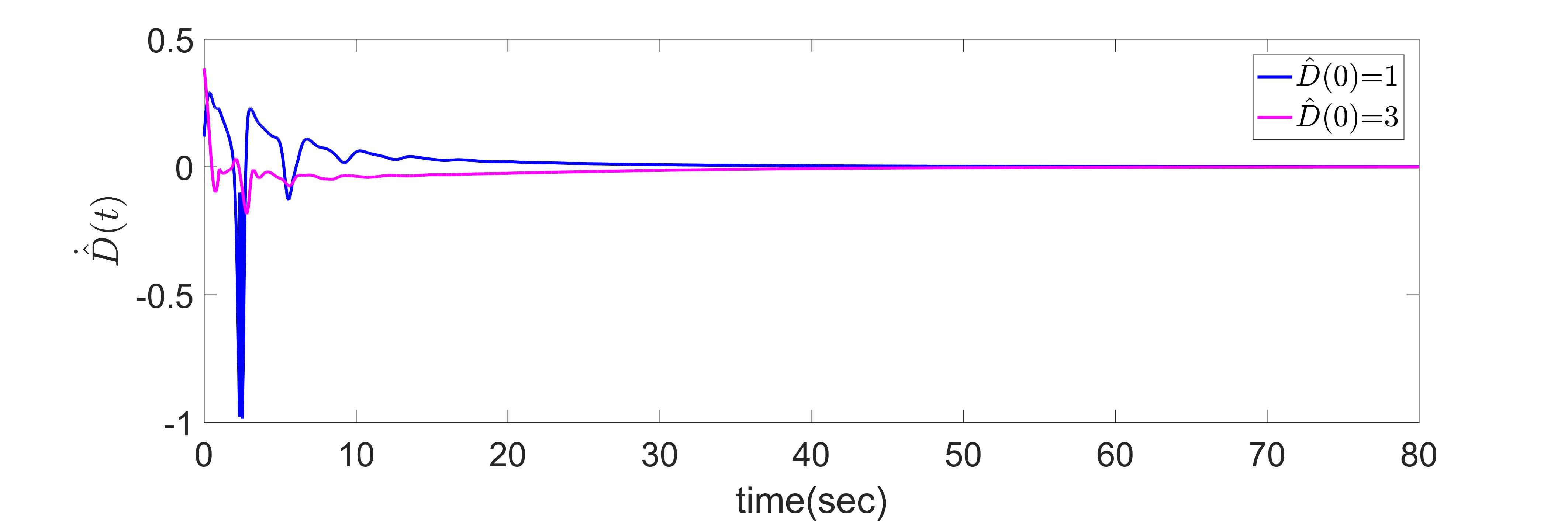}}
                \subfigure[]{\includegraphics[width=0.495\textwidth]{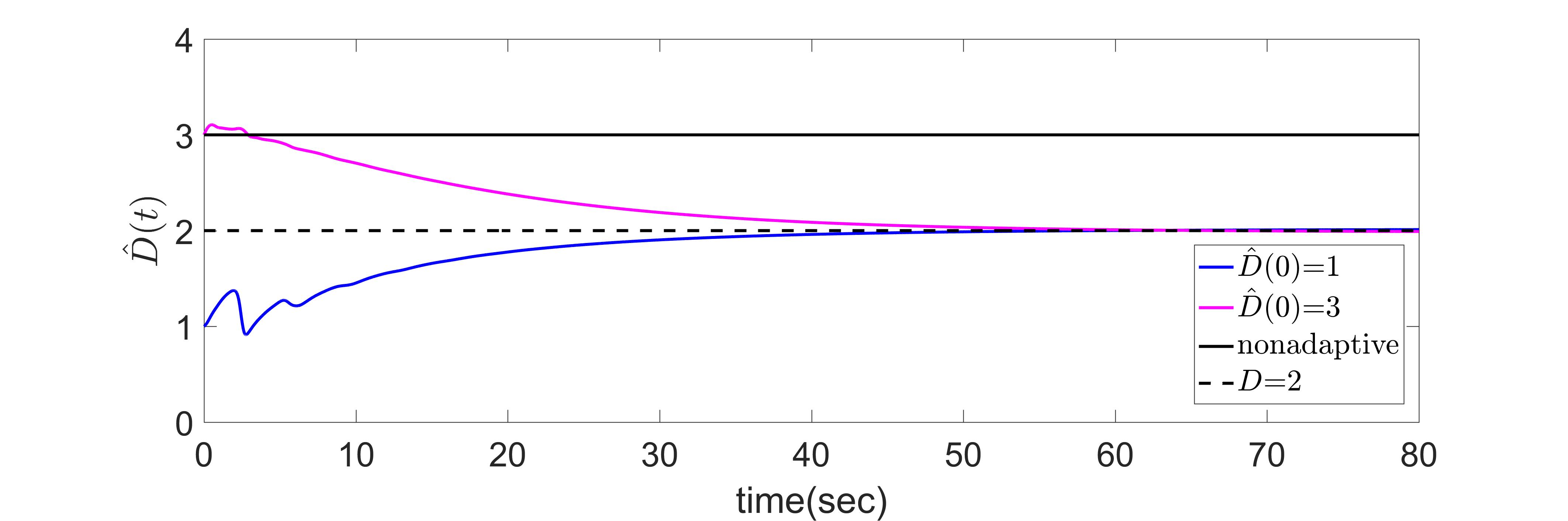}}
                \caption{The closed-loop system dynamics with $u_0(x)$, $v_0(x)$ and to $\hat D(0)$ with and without adaptation. (a) $L^2$-norm of the plant state $u(x,t)$. (b) The time evolution of the control signal.  (c) The dynamics of the update law $\dot{\hat D}(t)$. (d) The time-evolution of the estimate of the unknown parameter $\hat D(t)$.}\label{figure2}
        \end{figure}
        
        \section{Conclusion}\label{6}
        We have studied a class of first-order hyperbolic PIDEs systems with an input subject to an unknown time delay. By utilizing an infinite-dimensional representation of the actuator delay, the system was transformed into a cascading structure consisting of a transport PDE and a PIDE. We successfully established global stability results by designing a parameter update law using the well-known infinite-dimensional backstepping technique and a Lyapunov argument. Furthermore, we analyzed the well-posedness of the system, taking into account the added difficulty caused by the presence of nonlinear terms. Through numerical simulations, we have demonstrated the effectiveness of the proposed method. This research contributes to the understanding and control of systems with unknown time delays and provides valuable insights into the stability analysis and parameter update design for such systems. Future work may involve extending these findings to more complex systems or considering additional constraints and uncertainties.
        
        %\backmatter
        
        \section*{Acknowledgments}
        This work is partly supported by partially supported by National Natural Science Foundation of China (62173084, 61773112), the Natural Science Foundation of Shanghai (23ZR1401800).
        % 
        % \subsection*{Conflict of interest}
        % 
        % The authors declare no potential conflict of interests.
%       \bibliographystyle{plain}
        \bibliographystyle{ieeetr}
        \bibliography{Adaptive_TAC}  
\end{document}